\documentclass{amsart}
\input cyracc.def

\usepackage{amssymb}
\usepackage{amsthm}

\theoremstyle{plain}
\newtheorem{thm}{Theorem}[section]
\newtheorem{lem}[thm]{Lemma}

\newtheorem{prop}[thm]{Proposition}

\theoremstyle{definition}
\newtheorem{dfn}[thm]{Definition}

\newtheorem{prob}[thm]{Problem}

\newtheorem{ex}[thm]{Example}

\def\dnfo{\;\raise.2em\hbox{$\mathrel|\kern-.9em\lower.4em\hbox
{$\smile$}$}}

\def\dnf#1{\lower.9em\hbox{$\buildrel\dnfo\over{ \scriptstyle  #1}$}}

\def\dfo{\;\raise.2em\hbox{$\mathrel|\kern-.9em\lower.4em\hbox{$\smile$}
\kern-.72em\lower.07em\hbox{\char'57}$}\;}
\def\df#1{\lower1em\hbox{$\buildrel\dfo\over{\scriptstyle #1}$}}

\newcommand{\sC}{\mathcal{C}}
\newcommand{\sD}{\mathcal{D}}

\newcommand{\sP}{\mathcal{P}}

\newcommand{\C}{\mathbb{C}}
\newcommand{\rat}{\mathbb{Q}}

\title{PAC Learning, VC Dimension, and the Arithmetic Hierarchy}

\author{Wesley Calvert}
\address{Department of Mathematics\\ Mail Code 4408\\
Southern Illinois University, Carbondale\\
1245 Lincoln Drive\\
Carbondale, Illinois 62901}
\email{wcalvert@siu.edu}
\date{\today}

\begin{document}

\begin{abstract}
We compute that the index set of PAC-learnable concept classes is
$m$-complete $\Sigma^0_3$ within the set of indices for all concept
classes of a reasonable form.  All concept classes considered are
computable enumerations of computable $\Pi^0_1$ classes, in a sense
made precise here.  This family of concept classes is sufficient to
cover all standard examples, and also has the property that PAC
learnability is equivalent to finite VC dimension.
\end{abstract}

\maketitle

\section{Introduction}

A common method to characterize the complexity of an object is to
describe the degree of its index set
\cite{ecl2,d2eqs,d2apg,idxsets,idxpi01,gk,melnikovnies}.  In the present paper, we
carry out this computation for the class of objects which are
machine-learnable in a particular model.

There have been several models of machine learning, dating back at
least to Gold's seminal 1967 paper \cite{gold1967}.  In Gold's basic model,
the goal is that the machine should determine a $\Sigma^0_1$-index
for a computably enumerable set of natural numbers --- that is, an index for a computable function enumerating
it, by receiving an initial segment of the string.  Of course, many
variations are possible, involving, for instance, the receipt of positive or negative
information and the strength of the convergence criteria in the task
of ``determining'' an index.  This family of models has been studied
by the recursion theory community (see, for instance,
\cite{friendgoetheharizanov,harizanovstephan,stephanventsov}), but is
not the primary focus of this paper.  One particular result, however,
is of interest to us.

\begin{thm}[Beros \cite{Beros}] The set of $\Sigma^0_1$ indices for
  uniformly computably enumerable families learnable in each of the following
  models is $m$-complete in the corresponding class.
\begin{enumerate}
\item TxtFin --- $\Sigma^0_3$
\item TxtEx --- $\Sigma^0_4$
\item TxtBC --- $\Sigma^0_5$
\item TxtEx$^*$ --- $\Sigma^0_5$
\end{enumerate}\end{thm}

\subsection{PAC Learning}

The model of learning that concerns us here (\emph{PAC learning}, for
``Probably Approximately Correct'') was first proposed by
Valiant in \cite{Valiant}.  Much of our exposition of the subject comes from
\cite{KearnsVazirani}.  The idea of the model is that it should allow
some acceptably small error of each of two kinds: one arising from
targets to be learned which are somehow too close together to be
easily distinguished, and the other arising from randomness in the
examples shown to the learner.  Neither aspect is easily treated
in Gold's framework of identifying indices for computable enumerations
of natural numbers by inspecting initial segments --- neither a notion
of ``close'' nor randomness in the inputs.

In the present paper, we will describe a framework in which to model
PAC learning in a way which is suitable for recursion-theoretic
analysis and which is broad enough to include many of the benchmark
examples.  We will then calculate the $m$-degree of the set of indices
for learnable concept classes.

\begin{dfn}[Valiant] \rule{0in}{0in}
\begin{enumerate}
\item Let $X$ be a set, called the \emph{instance
    space}.  
\item Let $\sC$ be a subset of $\sP(X)$, called a \emph{concept
    class}.  
\item The elements of $\sC$ are called \emph{concepts}.  
\item We say
  that $\sC$ is \emph{PAC Learnable} if and only if there is an
  algorithm $\varphi_e$ such that for every $c \in \sC$, every $\epsilon,\delta \in (0,\frac{1}{2})$ and every probability
  distribution $\sD$ on $X$, the algorithm $\varphi_e$ behaves as follows:
On input $(\epsilon, \delta)$, the algorithm $\varphi_e$ will ask for some
  number $n$ of examples, and will be given $\{(x_1,i_1), \dots, (x_n,i_n)\}$
  where $x_j$ are independently randomly drawn from $\sD$, and $i_j =
  \chi_c(x_j)$.  The algorithm will then output some $h \in \sC$ so
  that with probability at least $1-\delta$ in $\sD$, the symmetric difference
  of $h$ and $c$ has probability at most $\epsilon$ in $\sD$.
\end{enumerate}
\end{dfn}

This is a well-studied model --- so well-studied, in fact, that it is
more usual to talk about the complexity of the algorithm (in both
running time and the number of example calls) than about its
existence.  For the present paper, though, we restrict ourselves to
the latter problem.  Several examples are well-known.

\begin{ex} Let $X = 2^n$, interpreted as assignments of truth values
  to Boolean variables.  Then the class $\sC$ of $k$-CNF expressions is
  PAC learnable (where each expression $c \in \sC$ is interpreted as the
  set of truth assignments that satisfy it).\end{ex}

\begin{ex} Let $X = \mathbb{R}^d$.  Then the class $\sC$ of linear
  half-spaces is PAC learnable.\end{ex}

\begin{ex} Let $X = \mathbb{R}^2$.  Then the class of convex $d$-gons
  is PAC learnable for any $d$.\end{ex}

\subsection{The Vapnik-Chervonenkis Dimension}

An alternate view of PAC learnability arises from work of Vapnik and
Chervonenkis \cite{VC}.  Again, we follow the exposition of
\cite{KearnsVazirani}.

\begin{dfn} Let $\sC$ be a concept class.
\begin{enumerate}
\item Let $S \subseteq X$.  Then $\Pi_\sC(S) = \left|\left\{S \cap c :
  c \in \sC\right\}\right|$.
\item The VC dimension of $\sC$ is the greatest integer $d$ such that
  $\Pi_\sC(S) = 2^d$ for some $S$ with cardinality $d$, if such an
  integer exists.  Otherwise, the VC dimension of $\sC$ is $\infty$.
\end{enumerate}
\end{dfn}

For example, if $\sC$ is the class of linear half-spaces of
$\mathbb{R}^2$, and if $S$ is a set of size 4, suppose that $k$ is the
least such that all of $S$ is contained in the convex hull of $k\leq
4$ points.  If $k<4$, take a set $S_0$ of size $k$ such that the
convex hull of $S_0$ contains $S$.  The subset $S_0 \subset S$ cannot
be defined by intersecting $S$ with a linear half-space.  If $k=4$,
then let $S_0$ be a diagonal pair, which again cannot be defined by
intersection with a linear half-space.  Consequently, the VC dimension
of $\sC$ must be at most 3.  One can also show that this bound is
sharp.

The connection of VC dimension with learnability is a theorem of
Blumer, Ehrenfeucht, Haussler, and Warmuth showing that under some
reasonable measure-theoretic hypotheses (which hold in all examples
shown so far, and in all examples that will arise in the present
paper), finite VC dimension is equivalent to PAC learnability
\cite{BEHW}.

\begin{dfn}[Ben-David, as described in \cite{BEHW}] Let $R \subseteq \sP(X)$, and let $\sD$ be a probability distribution on $X$, and $\epsilon >0$.  
\begin{enumerate}
\item We say that $N \subseteq X$ is an $\epsilon$-transversal for $R$
  with respect to $\sD$ if and only if for any $c \in R$ with
  $P_{\sD}(c) >\epsilon$ we have $N \cap R \neq \emptyset$.
\item For each $m \geq 1$, we denote by $Q^m_\epsilon(R)$ the set of
  $\vec{x} \in X^m$ such that the set of distinct elements of
  $\vec{x}$ does not form an $\epsilon$-transversal for $R$ with
  respect to $\sD$.
\item For each $m\geq 1$, we denote by $J^{2m}_{\epsilon}(R)$ the set of
  all $\vec{x}\vec{y} \in X^{2m}$ with $\vec{x}$ and $\vec{y}$ each of
  length $m$ such that there is $c \in R$ with $P_\sD(c)>\epsilon$
  such that no element of $c$ occurs in $\vec{x}$, but elements of $c$
  have density at least $\frac{\epsilon m}{2}$ in $\vec{y}$.
\item We say that a concept class $\sC$ is \emph{well-behaved} if for
  every Borel set $b$, the sets $Q^m_\epsilon(R)$ and $J^{2m}_\epsilon(R)$
  are measurable where $R = \{c \triangle b : c \in \sC\}$.
\end{enumerate}
\end{dfn}

This notion of ``well-behaved'' is exactly the necessary hypothesis
for the equivalence:

\begin{thm}[\cite{BEHW}]\label{wellbehaved} Let $\sC$ be a nontrivial, well-behaved
  concept class.  Then $\sC$ is PAC learnable if and only if
  $\sC$ has finite VC dimension.\end{thm}

\section{Concepts and Concept Classes}

The most general context in which PAC learning makes sense is far too
broad to say anything meaningful about the full problem of determining
whether a class is learnable.  If we were to allow the instance space
to be an arbitrary set, and a concept class an arbitrary subset of the
powerset of the instance space, we would quickly be thinking about a
non-trivial fragment of set theory.

In practice, on the other hand, one usually fixes the
instance space, and asks whether (or how efficiently, or just by what
means) a particular class is learnable.  This approach is too narrow
for the main problem of this paper to be meaningful.  The goal of this
section, then, is to describe a context broad enough to cover many of
the usual examples, but constrained enough to be tractable.

Many of the usual examples of machine learning problems can be
systematized in the framework of $\Pi^0_1$ classes, which will now be
introduced.  The following result is well-known, but a proof is given
in \cite{Pi01}, which is also a good general reference on $\Pi^0_1$
classes.

\begin{thm}\label{pi01eq} Let $c \subseteq 2^\omega$.  Then the following are
  equivalent:
\begin{enumerate}
\item $c$ is the set of all infinite paths through a computable subtree of
  $2^\omega$
\item $c$ is the set of all infinite paths through a $\Pi^0_1$ subtree of
  $2^\omega$ (i.e.\ a co-c.e.\ subtree)
\item $c = \{x \in 2^\omega : \forall n \ R(n,x)\}$ for some
  computable relation $R$, i.e.\ a relation $R$ for which there is a
  Turing functional $\Phi$ such that $R(n,x)$ is defined by
  $\Phi^x(n)$.
\end{enumerate}
\end{thm}

This equivalence (and other similar formulations could be added) gives
rise to the following definition:

\begin{dfn} Let $c \subseteq 2^\omega$.  Then we say that $c$ is a
  $\Pi^0_1$ class if and only if it satisfies one of the equivalent
  conditions in Theorem \ref{pi01eq}.\end{dfn}

\begin{ex} There is a natural and uniform representation of all
  well-formed formulas of classical propositional calculus, each as a
  $\Pi^0_1$ class.  We regard $2^\omega$ as the assignment of values
  to Boolean variables, so that for $f \in 2^\omega$, the value $f(n)
  = k$ indicates a value of $k$ for variable $x_n$.  Let $\varphi$ be
  a propositional formula.  We construct a $\Pi^0_1$ subtree
  $T_\varphi \subseteq 2^\omega$ such that $f \in T_\varphi$ if and
  only if $f$ satisfies $\varphi$.  At stage $n$, for each $\sigma \in
  2^{<\omega}$ of length $n$, we include $\sigma \in T_\varphi$ if and
  only if there is an extension $f \supset \sigma$ such that $F
  \models \varphi$.  This condition can be checked effectively.
  Consequently, $T_\varphi$ is a $\Pi^0_1$ subtree of $2^\omega$ ---
  intuitively, an infinite path $f$ may fall out of $T_\varphi$ at
  some point when we see a long enough initial segment to detect
  non-satisfiability, but unless it falls out at some finite stage, it
  is included.\end{ex}

\begin{ex} There is a natural and uniform representation of all
  \emph{closed intervals of $\mathbb{R}$} with computable endpoints, each as
  a $\Pi^0_1$ class.  We take the usual representation of real numbers
  by binary strings.  Let $I$ be a closed interval with computable
  endpoints.  We construct a $\Pi^0_1$ tree $T_I \subseteq 2^\omega$
  such that the set of paths through $T_I$ is equal to $I$.  At stage
  $s$, we include in $T_I$ all binary sequences $\sigma$ of length $s$
  such that there is an extension $f \supset \sigma$ with $f \in I$.
  This condition can be checked effectively, by the computability of
  the endpoints of $I$.  Consequently, $T_I$ is a $\Pi^0_1$ subtree of
  $2^\omega$.
\end{ex}

\begin{ex} There is a natural and uniform representation of all
  \emph{closed linear
  half-spaces of $\mathbb{R}^d$} which are defined by hyperplanes with
  computable coefficients, each half-space as a $\Pi^0_1$ class.  We encode
  $\mathbb{R}^d$ as $2^\omega$ in the following way: the $i$th
  coordinate of the point represented by the path $f$ is given by the
  sequence $\left(f(k) : k \equiv i \mod d\right)$.  Now we encode a
  linear subspace into a subtree in the same way as with intervals in
  the previous example.
\end{ex}

\begin{ex} There is a natural and uniform representation of all
  \emph{convex $d$-gons in $\mathbb{R}^2$} with computable vertices,
  with each $d$-gon represented by a $\Pi^0_1$ class.  A convex
  $d$-gon is an intersection of $d$ closed linear half-spaces, and so
  we exclude a node $\sigma \in 2^\omega$ from the tree for our
  $d$-gon if and only if it is excluded from the tree for at least one
  of those linear half-spaces.\end{ex}

Note that the requirement of computable boundaries of these examples
is not a practical restriction.  

\begin{prop}\label{approx} For any probability measure $\mu$ on $\mathbb{R}^d$
  absolutely continuous with respect to Lebesgue measure, and for any
  hyperplane given by $f(\vec{x}) = 0$, there is a hyperplane given by
  $\bar{f}(\vec{x}) = 0$ where $\bar{f}$ has computable coefficients,
  and where the linear half-spaces defined by these hyperplanes are
  close in the following sense: If $H_f$ is defined by $f(\vec{x})
  \leq 0$, if $H^0_f$ is defined by $f(\vec{x}) < 0$, and $H_{\bar{f}}$
  is defined by $\bar{f}(\vec{x})\leq 0$, then $\mu\left(H_f \triangle
  H_{\bar{f}}\right) < \epsilon$ and $\mu\left(H^0_f \triangle
  H_{\bar{f}}\right) <\epsilon$.\end{prop}

\begin{proof}
Since a hyperplane has Lebesgue measure 0, it suffices to show that we
can achieve $\mu\left(H_f \triangle H_{\bar{f}}\right) < \epsilon$.
Now by using the cumulative distribution function, we can construct a
bounded $d$-orthotope $B \subseteq \mathbb{R}^d$ with computable
vertices such that $\mu(B) \geq \frac{\epsilon}{2}$.

Since computable points are dense in $\mathbb{R}$, we can find, in
each face $F_i$ of $B$, a computable point $\vec{a}_i$ so close to
$f(\vec{x}) = 0$ that if $\bar{f}(\vec{x}) = 0$ is the hyperplane determined
by the set of points $\{\vec{a}_i : i \leq d\}$,
then \[\mu\left(\left(H_f \triangle H_{\bar{f}}\right) \cap B\right) <
\frac{\epsilon}{2}.\]  The coefficients of $\bar{f}$ are computable
since the points $\vec{a}_i$ are computable.  Furthermore, \[\mu\left(H_f \triangle H_{\bar{f}}\right) < \epsilon.\]\end{proof}

Examples could be multiplied, of course, and it seems likely that many
of the more frequently encountered machine learning situations could
be included in this framework --- certainly, for instance, any example
in \cite{bishop}, \cite{KearnsVazirani}, or \cite{russellnorvig}.

We will work, for the purposes of the present paper, with instance
space $2^\omega$ and with concepts which are $\Pi^0_1$ classes.  It
remains to describe the concept classes to be used.

There is an unfortunate clash of terminology in that the concept
\emph{classes} will have, for their members, $\Pi^0_1$
\emph{classes}.  In this paper, we will never use the term
ambiguously, but because both terms are so well-established it will be
necessary to use both of them.

\begin{dfn} A weakly effective concept class is a computable enumeration
  $\varphi_e : \mathbb{N} \to \mathbb{N}$ such that $\varphi_e(n)$ is
  a $\Pi^0_1$ index for a $\Pi^0_1$ tree $T_{e,n}$.
\end{dfn}

Naturally, we interpret each index enumerated as the $\Pi^0_1$ class
of paths through the associated tree.  We also freely refer to the
indices (or trees, or $\Pi^0_1$ classes) in the range of a concept
class as its elements.

This definition is almost adequate to our needs.  We would like,
however, one additional property: that a finite part of an effective
concept class should not be able to distinguish a non-computable point
of $2^\omega$ from all computable points.  This is reasonable: it
would strain our notion of an ``effective'' concept class if it should
fail.  And yet it can fail with a weakly effective concept class: our
classes may have no computable members at all, for instance.  For that
reason, we define an effective concept class as follows.

\begin{dfn} An effective concept class is a weakly effective concept
  class $\varphi_e$ such that for each $n$, the set $c_n$ of paths through
  $T_{e,n}$ is computable in the sense that there is a computable
  function $f_{c_n}(d,r):2^{<\omega} \times \rat \to \{0,1\}$
  such that \[f_{c_n}(\sigma, r) = \left\{\begin{array}{ll} 1 &
  \mbox{ if $B_r(\sigma) \cap c_n \neq \emptyset$}\\
0 & \mbox{ if $B_{2r}(\sigma) \cap c_n = \emptyset$}\\
\mbox{$0$ or $1$} & \mbox{ otherwise}\\
\end{array}\right.\]
where $B_r(\sigma)$ is the set of all paths that either extend
$\sigma$ or first differ from it at the $-\lceil lg(r)\rceil$ place or
later (see
\cite{bravermanyampolskybk,weihrauch}).
\end{dfn}

In addition to the useful property mentioned above, which we will soon
prove, there is another reason for preferring this stronger
definition: Typically when we want a computer to learn something, it
is with the goal that the computer will then be able to act on it.
Computability of each concept is a necessary condition for this.  The restriction
corresponds, in the examples, to the restriction that a linear
half-space, for instance, be defined by computable coefficients.  The
classes we consider in this paper will be effective concept classes.

\begin{prop}\label{comprepl} Let $\sC$ be an effective concept class, and let $c_1,
  \dots, c_k \in \sC$.  Then for any $y \in 2^\omega$, there is a
  computable $x \in 2^\omega$ such that for each $i \in \{1, \dots,
  k\}$, we have $x \in c_i$ if and only if $y \in c_i$.\end{prop}

\begin{proof} Let $y, c_1, \dots, c_k$ be as described in the
  statement of the Proposition.  Let $I$ be the set of $i$ such that
  $y \in c_i$ and $J$ be the set of $i$ such that $y \notin c_i$.

Suppose first that $y$ is not in the boundary $\partial c_i$ of $c_i$
for each $i$.  Then \[y \in N := \left(\bigcap\limits_{i \in I}
c_i^\circ\right) \cap \left(\bigcap\limits_{i \in J}
(\overline{c_i})^\circ\right),\] where $S^\circ$ denotes the interior of $s$ and
$\bar{S}$ the complement of $S$.  Since $I$ and $J$ are finite, $N$ is
open.  Since $y \in N$, the set $N$ is nonempty, and must contain a
basic open set of $2^\omega$, and so must contain a computable member,
$x$, as required.

Now suppose that $y$ is in $\partial c_i$ for some $i$.  Then we can
compute $y$, using the function $f_{c_i}$, so that $y$ is itself
computable and we take $x = y$.\end{proof}

We note that all of the examples given so far are effective concept
classes.

\begin{ex} The class of well-formed formulas of classical
  propositional calculus, and the class of $k$-CNF expressions (for
  any $k$) are effective concept classes, by the example
  above.  Whether a given $y \in 2^\omega$ satisfies a particular
  formula can be determined by examining only finitely many terms of $y$.
\end{ex}

\begin{ex} The class $\sC$ of linear half-spaces in $\mathbb{R}^d$ bounded by
  hyperplanes with computable coefficients is an effective concept
  class.  Recall that each linear half-space with
  computable coefficients is a computable set, since the distance of
  a point from the boundary can be computed.\end{ex}

\begin{ex} The class of convex $d$-gons in $\mathbb{R}^2$ with
  computable vertices is an effective concept class.\end{ex}

Again, it appears that any example in any of the standard
references is an effective concept class.

A pleasant feature of the effective concept classes is that they are
always well-behaved.

\begin{lem} A weakly effective concept class has finite VC dimension if and
  only if it is PAC learnable.\end{lem}

\begin{proof} Let $\sC$ be an effective concept class.  In
  \cite{BEHW}, a proof of Ben-David is given that if $\sC$ is
  universally separable --- that is, if there is a countable subset
  $\sC^*$ such that every point in $\sC$ can be written as the
  pointwise limit of some sequence in $\sC^*$ --- then $\sC$ is
  well-behaved.  Since an effective concept class is always countable
  (i.e.\ it contains only countably many $\Pi^0_1$ classes),
  $\sC$ is trivially universally separable.  By Theorem
  \ref{wellbehaved}, the conclusion holds.\end{proof}

\section{Bounding the Degree of the Index Set}

We now turn toward the main problem of the paper, which we can now
express exactly.

\begin{prob} Determine the $m$-degree of the set of all natural
  numbers $e$ such that $\varphi_e$ is a PAC-learnable effective
  concept class.\end{prob}

One minor refinement in the problem remains: the difficulty of saying
that $e$ is the index for an effective concept class competes with
that of saying that this concept class is learnable.  Indeed, since
determining that $n$ is an $X$-index for an $X$-computable tree is
$m$-complete $\Pi^0_2(X)$ (see \cite{gk,soare}), it follows that determining
that $n$ is a $\Pi^0_1$ index for a $\Pi^0_1$ tree is $m$-complete
$\Pi^0_3$.

Since we will see that finite VC dimension can be defined at
$\Sigma^0_3$, a driving force in the $m$-degree described in the
problem above will be that it must compute all $\Pi^0_3$ sets.  This
tells us nothing about the complexity of learnability, but only about
the complexity of determining whether we have a concept class.  The
usual way to deal with this issue is by the following definition.

\begin{dfn}[\cite{ecl2}] Let $A \subseteq B$, and let $\Gamma$ be some class of
  sets (e.g.\ $\Pi^0_3$).
\begin{enumerate}
\item We say that \emph{$A$ is $\Gamma$ within $B$} if and only if $A
  = R \cap B$ for some $R \in \Gamma$.
\item We say that \emph{$S \leq_m A$ within $B$} if and only if there
  is a computable $f: \omega \to B$ such that for all $n$ we have $n
  \in S \Leftrightarrow f(n) \in A$.
\item We say that \emph{$A$ is $m$-complete $\Gamma$ within $B$} if
  and only if $A$ is $\Gamma$ within $B$ and for every $S \in \Gamma$
  we have $S \leq_m A$ within $B$.
\end{enumerate}
\end{dfn}

We can now present the question in its final form.

\begin{prob} Let $L$ be the set of indices for effective concept
  classes, $K$ the set of indices for effective concept classes which are PAC
  learnable.  What
  is the $m$-degree of $K$ within $L$?\end{prob}

The solution to the problem will have two parts.  In the present
section, we will show that $K$ is $\Sigma^0_3$ within $L$.  In the
following section, we show that $K$ is $m$-complete $\Sigma^0_3$
within $L$.

We first reduce the problem to one on computable paths through $2^\omega$.

\begin{prop} An effective concept class $\sC$ has infinite VC
  dimension if and only if for every $d$ there are (not necessarily
  uniformly) computable elements \[\left(x_i: i <d \right)\] such that $\Pi_\sC\left(x_i: i <d \right) = 2^d$..\end{prop}

\begin{proof}
Let $\left(y_i : i <d\right)$ witness that $\sC$ has VC dimension at
least $d$, and denote by $D_1, \dots
D_{2^d}$ elements of $\sC$ which distinguish distinct subsets of
$\left(y_i : i <d\right)$.  For each $i<d$, there is a computable
element $x_i$ such that for every $j \leq 2^d$ we have $x_i \in D_j$
if and only if $y_i \in D_j$, by Proposition \ref{comprepl}.  Then
$x_1, \dots x_d$ witness that $\sC$ has VC dimension at least $d$.
The converse is obvious. \end{proof}

\begin{prop} The set of indices for effective concept classes of
  infinite VC dimension is $\Pi^0_3$ within $L$.\end{prop}

\begin{proof}
We begin by noting that if $f$ is a computable function and $T$
is a $\Pi^0_1$ tree, then it is a $\Pi^0_1$ condition that $f$ is a
path of $T$, and a $\Sigma^0_1$ condition that it is not, uniformly in
a $\Pi^0_1$ index for $T$ and a computable index for $f$.  Further, if
$\sC = \varphi_e$ is an effective concept class, then for any $k \in
\omega$, the condition that $k \in ran(\varphi_e)$ is a $\Sigma^0_1$
condition, uniformly in $e$ and $k$. 

Let $(x_1,
\dots, x_n)$ be a sequence of computable functions, $S \subseteq
\{1,2, \dots, n\}$, and $c$ a $\Pi^0_1$ class, represented by a
$\Pi^0_1$ index for a tree in which it is the set of paths.  We
abbreviate by $c \upharpoonright_n = S$ the statement that for each $i
\in \{1,\dots n\}$, we have $x_i \in c$ if and only if $i \in S$.  Now
$c \upharpoonright_n = S$ is a $d$-$\Sigma^0_1$ condition, uniformly
in the indices for the $x_i$ and $c$.

We now note that $\sC = \varphi_e$ has infinite VC dimension if and
only if \[\bigwedge\limits_{n \in \mathbb{N}}\hspace{-0.15in}\bigwedge
\exists x_1, \dots, x_n \bigwedge\limits_{S \subseteq (n+1)} \exists k
\left[\varphi_e(k)\upharpoonright_n = S\right] .\]  From the comments
above, this definition is $\Pi^0_3$.
\end{proof}

\section{Sharpness of the Bound}

The completeness result in this section will finish our answer to the
main question of the paper.  

\begin{thm} The set of indices for effective concept classes of
  infinite VC dimension is $m$-complete $\Pi^0_3$ within $L$, and the
  set of indices for effective concept classes of finite VC dimension
  is $\Sigma^0_3$ within $L$.\end{thm}

\begin{proof} It only remains to show completeness.  For each
  $\Pi^0_3$ set $S$, we will construct a sequence of effective concept
  classes $\left(\sC_n : n \in \mathbb{N}\right)$ such that $\sC_n$
  has infinite VC dimension if and only if $n \in S$.  In the
  following lemma, to simplify notation, we suppress the dependence of
  $f$ on $n$.

\begin{lem} There is a $\Delta^0_2$ function $f:\mathbb{N} \to 2$ such
  that $f(s) = 1$ for infinitely many $s$ if and only if $n \in
  S$.\end{lem}

\begin{proof} 
It suffices (see \cite{soare}) to consider $S$ of the form $\exists^\infty x \forall y
R(x,y,n)$.  Now we set \[f(x) = \left\{\begin{array}{ll} 1 & \mbox{if
  $\forall y R(x,y,n)$}\\ 0 &
\mbox{otherwise}\\ \end{array}\right. .\]
This function is $\Delta^0_2$-computable, and has the necessary properties.
\end{proof}

Now by the Limit Lemma, there is a uniformly computable sequence $\left(f_s : s \in
\mathbb{N}\right)$ of functions such that for each $x$, for
sufficiently large $s$, we have $f_s(x) = f(x)$.

We now take a set of functions that will serve as the elements that
may eventually witness high VC dimension.  Let $\left\{\pi_{s,t,j} :
s,t,j \in \mathbb{N}, j<s\right\}$ be a discrete uniformly computable set of
distinct elements of $2^\omega$ such that $\pi_{s,t,j}(q) =
\pi_{s,t',j'}(q)$ whenever $q < \min\{t,t'\}$.  

We also initialize $G_{s,0} = \emptyset$ for each $s$.  Denote by
$P_t$ a bijection \[P_t:\sP\left(\{1, \dots, t\}\right) \to \{1, \dots, 2^t\}.\]

At stage $s$ of the construction, we consider $f_s(t)$ for each $t\leq
s$.  If $f_s(t) = 0$, then no action is required.

If $f_s(t) = 1$,
then we find the least $k$ such that $k \notin G_{t,s}$.  Let
$\{e_{t,i} : i<2^t\}$ be $\Pi^0_1$ indices for trees such that
$T_{e_{t,i}}$ consists exactly of the initial segments $\tau$ of
$\pi_{t,k,j}$ where $j = P_t(S)$ for some $S \subseteq \{1, \dots,
t\}$ and $|\tau|$ is less than the
first $z>s$ such that $f_z(t) = 0$.  This can be done effectively
exactly because we are looking for $\Pi^0_1$ indices, and the search
is uniform.  We then let $i_s$ be the least such that $\sC_n(i_s)$ is
undefined, and take $\sC_n(i_s+\ell) = e_{t,\ell}$ for each
$\ell<2^t$.  We also set $G_{t,s} = G_{t,s-1} \cup k$.

Now for each $t$ with $f(t) = 1$, there will be some $s$ such that
$f_{s'}(t) = f_s(t) = 1$ for all $s'>s$.  Then at stage $s$ we have
added to $\sC_n$ the $\Pi^0_1$ indices $\{e_{t,i} : i <2^t\}$
guaranteeing that $\{\pi_{t,k,j} : j<t\}$ is shattered for some $k$.

For each $t$ such that $f(t) = 0$ and each $s$ such that $f_s(t) = 1$,
there is some later stage $s'$ such that $f_{s'}(t) = 0$, so any
indices added at stage $s$ will be indices for a tree with no paths
--- that is, for the empty concept.

Note that if the same $t$ receives attention infinitely often --- that
is, if infinitely many different sets of classes are added to $\sC_n$
to guarantee that the VC dimension of $\sC_n$ is at least $t$, this
does not inflate the VC dimension beyond $t$.  Indeed, the sets of
witnesses will be pairwise disjoint, so no concept in $\sC_n$ will
include any mixture of witnesses from different treatments; the
resulting sets will not be shattered.

We further note that all the $\Pi^0_1$ classes in $\sC_n$ are
computable.  Indeed, each $c \in \sC_n$ consists of finitely many
(perhaps no) computable paths.  Thus, $\C_n$ is an effective concept
class.

Now if $n \notin S$, then $f(s) = 1$ for at most finitely many $s$, so
that the VC dimension of $\sC_n$ is finite.  If $n \in S$, then $f(s)
= 1$ for infinitely many $s$, so that the VC dimension of $\sC_n$ is
infinite (since sets of arbitrarily large size will be shattered).
\end{proof}

\bibliographystyle{amsplain}
\bibliography{pac}

\providecommand{\bysame}{\leavevmode\hbox to3em{\hrulefill}\thinspace}
\providecommand{\MR}{\relax\ifhmode\unskip\space\fi MR }
\providecommand{\MRhref}[2]{%
  \href{http://www.ams.org/mathscinet-getitem?mr=#1}{#2}
}
\providecommand{\href}[2]{#2}
\begin{thebibliography}{10}

\bibitem{Beros}
A.~Beros, \emph{Learning theory in the arithmetic hierarchy}, preprint, 2013.

\bibitem{bishop}
C.~M. Bishop, \emph{Pattern recognition and machine learning}, Information
  Science and Statistics, Springer, 2006.

\bibitem{BEHW}
A.~Blumer, A.~Ehrenfeucht, D.~Haussler, and M.~K. Warmuth, \emph{Learnability
  and the {V}apnik-{C}hervonenkis dimension}, Journal of the ACM \textbf{36}
  (1989), 929--965.

\bibitem{bravermanyampolskybk}
M.~Braverman and M.~Yampolsky, \emph{Computability of julia sets}, Algorithms
  and Computation in Mathematics, no.~23, Springer, 2009.

\bibitem{ecl2}
W.~Calvert, \emph{The isomorphism problem for computable {A}belian $p$-groups
  of bounded length}, Journal of Symbolic Logic \textbf{70} (2005), 331--345.

\bibitem{d2eqs}
W.~Calvert, D.~Cenzer, V.~Harizanov, and A.~Morozov, \emph{Effective
  categoricity of equivalence structures}, Annals of Pure and Applied Logic
  \textbf{141} (2006), 61--78.

\bibitem{d2apg}
\bysame, \emph{{$\Delta^0_2$}-categoricity of {A}belian $p$-groups}, Annals of
  Pure and Applied Logic \textbf{159} (2009), 187--197.

\bibitem{idxsets}
W.~Calvert, V.~Harizanov, J.~F. Knight, and S.~Miller, \emph{Index sets of
  computable structures}, Algebra and Logic \textbf{45} (2006), 306--325.

\bibitem{Pi01}
D.~Cenzer, \emph{{$\Pi^0_1$} classes in computability theory}, Handbook of
  Computability, Studies in Logic and the Foundations of Mathematics, no. 140,
  Elsevier, 1999, pp.~37--85.

\bibitem{idxpi01}
D.~Cenzer and J.~Remmel, \emph{Index sets for {$\Pi^0_1$} classes}, Annals of
  Pure and Applied Logic \textbf{93} (1998), 3--61.

\bibitem{friendgoetheharizanov}
M.~Friend, N.~B. Goethe, and V.~Harizanov, \emph{Induction, algorithmic
  learning theory, and philosophy}, Logic, Epistemology, and the Unity of
  Science, vol.~9, Springer, 2007.

\bibitem{gold1967}
E.~M. Gold, \emph{Language identification in the limit}, Information and
  Control \textbf{10} (1967), 447--474.

\bibitem{gk}
S.~S. Goncharov and J.~F. Knight, \emph{Computable structure and non-structure
  theorems}, Algebra and Logic \textbf{41} (2002), 351--373.

\bibitem{harizanovstephan}
V.~Harizanov and F.~Stephan, \emph{On the learnability of vector spaces},
  Journal of Computer and System Sciences \textbf{73} (2007), 109--122.

\bibitem{KearnsVazirani}
M.~J. Kearns and U.~V. Vazirani, \emph{An introduction to computational
  learning theory}, MIT Press, 1994.

\bibitem{melnikovnies}
A.~G. Melnikov and A.~Nies, \emph{The classification problem for compact
  computable metric spaces}, The Nature of Computation: Logic, Algorithms,
  Applications, Lecture Notes in Computer Science, vol. 7921, Springer, 2013,
  pp.~320--328.

\bibitem{russellnorvig}
S.~Russell and P.~Norvig, \emph{Artificial intelligence}, 3rd ed., Prentice
  Hall, 2010.

\bibitem{soare}
R.~I. Soare, \emph{Recursively enumerable sets and degrees}, Springer-Verlag,
  1987.

\bibitem{stephanventsov}
F.~Stephan and Y.~Ventsov, \emph{Learning algebraic structures from text},
  Theoretical Computer Science \textbf{268} (2001), 221--273.

\bibitem{Valiant}
L.~G. Valiant, \emph{A theory of the learnable}, Communications of the ACM
  \textbf{27} (1984), 1134--1142.

\bibitem{VC}
V.~N. Vapnik and A.~Ya. Chervonenkis, \emph{On the uniform convergence of
  relative frequencies of events to their probabilities}, Theory of probability
  and its applications \textbf{16} (1971), 264--280.

\bibitem{weihrauch}
K.~Weihrauch, \emph{Computable analysis}, Texts in Theoretical Computer
  Science, Springer, 2000.

\end{thebibliography}

\end{document}